\documentclass[12pt,final]{amsart}
\usepackage{amsmath,amssymb}

\usepackage{cite}

\newtheorem{lemma}{Lemma}[section]
\newtheorem{theorem}[lemma]{Theorem}

\newtheorem{proposition}[lemma]{Proposition}

\newtheorem{example}[lemma]{Example}

\begin{document}
\begin{center}
{\Large \bf Quantitative results on continuity of the spectral factorization mapping}\\[10mm]Lasha Ephremidze$^{1,3}$,  Eugene Shargorodsky$^2$, and  Ilya Spitkovsky$^{1}$\\[5mm]
{\small $^1$ Division of Science and Mathematics, New York University Abu Dhabi,  UAE\\
$^2$ Department of Mathematics, King's College London,  UK\\
$^3$ Razmadze Mathematical Institute of Tbilisi State University, Georgia\\ [8mm]}
{\bf Abstract}
\end{center}
\vskip+0.5cm
{\small The spectral factorization mapping $F\to F^+$ puts a positive definite integrable matrix function $F$ having an integrable
 logarithm of the determinant in correspondence with an outer analytic matrix function $F^+$ such that $F = F^+(F^+)^*$ almost everywhere. The main question addressed here is to what extent $\|F^+ - G^+\|_{H_2}$ is controlled by $\|F-G\|_{L_1}$ and $\|\log \det F - \log\det G\|_{L_1}$.}

 \vskip+0.5cm

 \section{Introduction}
\label{intro}

Let $\mathcal{S}_n(\mathbb{T})$ be the class of $n\times n$ matrix spectral densities with integrable logarithms of the determinant, i.e. $F\in\mathcal{S}_n(\mathbb{T})$ if and only if
$F$ is a positive definite (a.e. on $\mathbb{T}=\{z\in\mathbb{C}:|z|=1\}$)  $n\times n$ matrix function with integrable entries, $F\in L_1(\mathbb{T})^{n\times n}$, such that the Paley-Wiener condition
\begin{equation}\label{PW}
\log \det F\in L_1({\mathbb T})
\end{equation}
is satisfied. Matrix spectral factorization theorem \cite{Wie57}, \cite{HelLow58} asserts that if $F\in\mathcal{S}_n(\mathbb{T})$,
then $F$ admits a factorization
\begin{equation}\label{MSF}
    F(t)=F^+(t)\big(F^+(t)\big)^*,
\end{equation}
where $F^+$ can be extended inside $\mathbb{T}$ to an outer analytic $n\times n$ matrix function with entries from the Hardy space $H_2$, $F^+\in H_2^{n\times n}$. Here and below $M^*$ stands for the Hermitian conjugate of any matrix $M\in\mathbb{C}^{n\times n}$. The spectral factor $F^+$ is unique up to a constant (right) unitary multiplier and thus it will be unique if we assume that $F^+(0)>0$. Hence we always assume that the latter condition is satisfied so that $F^+$ is uniquely defined.

In the scalar case ($n=1$), the spectral factor is given by the following formula
\begin{equation}\label{sc}
 f^+(z)=\exp\left(\frac 1{4\pi}
\int\nolimits_0^{2\pi}\frac{e^{i\theta}+z}{e^{i\theta}-z}\log
f(e^{i\theta})\,d\theta\right),\;\;\;f\in\mathcal{S}_1(\mathbb{T}).
\end{equation}
However, there is no explicit formula for $F^+$ in the matrix case.

Representation \eqref{MSF} plays a crucial role in the study of systems of singular integral equations \cite{MR0102720}, \cite{MR2663312}, in linear estimation \cite{Kai99}, quadratic and $H_\infty$ control \cite{MR0335000}, \cite{MR932459}, \cite{MR2663312}, communications \cite{fisher},  filter design \cite{MR1162107}, \cite{MR1411910},  etc. Recently, matrix spectral factorization became an important step in non-parametric estimations of Granger causality used in neuroscience \cite{Dhamala},\cite{RoyalA}.
In many of these applications, the spectral density function $F$ is usually constructed empirically and hence it is always subject to noise and measurement errors. Therefore, it is  important to know how close $\hat{F}^+$ remains to $F^+$ when we approximate $F$ by $\hat{F}$.
To the best of our knowledge, this problem has been investigated only in the scalar case so far  (see \cite{ESS1} and the references therein) and this is the first attempt in the matrix case to estimate the norm of $F^+-\hat{F}^+$ in terms of $F$ and $\hat{F}$.

Among several equivalent options to select the norm for measurement of the closeness, we deal with the operator norm of matrices:
\begin{equation}\label{nmm}
\|M\|=\sup_{\|x\|=1}\|Mx\|,
\end{equation}
where $x\in\mathbb{C}^{n\times 1}$ and $\|x\|$ is the Euclidian norm in $\mathbb{C}^n$, while for measurable matrix functions $F:\mathbb{T}\to\mathbb{C}^{n\times n}$, we use the norm
\begin{equation}\label{nmf}
\|F\|_{L_p}=\left(\frac1{2\pi}\int_0^{2\pi}\|F(e^{i\theta})\|^p\,d\theta\right)^{1/p},\;\;\;p\geq 1,
\end{equation}
with the standard convention for $p=\infty$: $\|F\|_{L_\infty}=\operatorname{ess} \sup\|F(e^{i\theta})\|$. The class of $n\times n$ matrix functions with finite $p$-norm will be denoted by $L_p(\mathbb{T})^{n\times n}$. It is assumed that $H_2^{n\times n}=H_2(\mathbb{D})^{n\times n}$, the Hardy space of analytic matrix functions, is isometrically embedded in $L_2(\mathbb{T})^{n\times n}$.

It is known that spectral factorization is not stable in $L_1$ norm in the sense that  convergence  $\|F_k-F\|_{L_1}\to 0$ does not guarantee $\|F_k^+-F^+\|_{H_2}\to 0$. However, as it was proved in \cite{Barclay} and \cite{EJL2011}, the latter convergence holds if we in addition require that $\|\log\det F_k-\log\det F\|_{L_1}\to 0$. Accordingly, it is reasonable to search for estimates of $\|F^+-\hat{F}^+\|_{H_2}$ in terms of $\|F-\hat{F}\|_{L_1}$ and $\|\log\det F-\log\det \hat{F}\|_{L_1}$.

The following results have been obtained in this direction for the scalar case (see \cite{ESS1} for more precise statements).

\begin{theorem}{\rm \cite[p. 520]{ESS1}}
There is no function
$\Pi : [0, +\infty)^2 \to [0, +\infty)$ such that
$\lim\limits_{s, t \to 0} \Pi(s, t) = 0$ for which the estimate
\begin{equation}\label{th.1.1.eq}
\|g^+ - f^+\|_{2}^2 \leq\Pi\left(\|g - f\|_{1}, \|\log g - \log f\|_{1}\right)
\end{equation}
holds for all $f, g \in\mathcal{S}_1(\mathbb{T})$ with $\|f\|_{1}, \|g\|_{1} \leq 1$.
\end{theorem}

On the other hand, if one allows the function $\Pi$ above to depend on $f$, then it is possible to prove estimate \eqref{th.1.1.eq} with a function
$\Pi=\Pi_f$ expressed in terms of Orlicz functions $\Phi$ for which $f\in L_\Phi$ (to be more precise, in terms of the complementary to $\Phi$ function  $\Psi$; see the definitions in Sect. 2 and Theorem \ref{ESS.Th3} below). These estimates take a more transparent form in the case $f\in L_p$:

\begin{theorem}{\rm \cite[Corollary 1]{ESS1}}\label{Thrm.2}
For every $p>1$,
\begin{equation}\label{estTh2}
\|f^+ - g^+\|_{H_2}^2 \le 2\|f - g\|_{L_1} + C(p) \|f\|_{L_p} \|\log f - \log g\|_{L_1}^{\frac{p - 1}{p}}\, ,
\end{equation}
where $C(p) = 2^{\frac{p + 1}{p}}\left(\frac{5p}{4(p - 1)}\right)^{\frac{p - 1}{p}}$.
\end{theorem}
The proof of the main result in \cite{ESS1} (formulated in Section 2 as Theorem \ref{ESS.Th3}) relied on a careful examination of the properties of the harmonic conjugation operator (the Hilbert transform) $f\mapsto\tilde{f}$ together with the theory of Orlicz spaces.

It is clear that an analogue of Theorem 1 is automatically correct in the matrix case and one is naturally led to the question  whether an  estimate of the form
\begin{equation}\label{3.6.1}
\|G^+ - F^+\|_{2}^2 \leq\Pi_F\left(\|G - F\|_{1}, \|\log \det G - \log \det F\|_{1}\right)
\end{equation}
holds in the matrix case where, as in the scalar case, $\Pi_F$ is expressed in terms of Orlicz functions $\Phi$ for which $F\in L_\Phi$.

In this paper, developing further the ideas of \cite{ESS1} and relying on the proof of the previously mentioned convergence theorem  presented in \cite{Barclay}, we provide  a positive answer to this question (see Remark after Theorem \ref{Th.3.1}), which yields estimates similar to \eqref{estTh2} in the matrix case. Since the problem is more difficult in the matrix case, the obtained estimates are more cumbersome.

Throughout the paper, $K$ denotes the best constant in Kolmogorov's weak type $(1, 1)$ inequality
$$
\left|\{\vartheta \in [-\pi, \pi) : \ |\widetilde{\psi}(\vartheta)| \ge \lambda\}\right| \leq \frac{K}{\lambda}\int_{-\pi}^\pi|\psi(\vartheta)|\,d\vartheta,\;\;\lambda>0,\;\;\psi\in L_1[-\pi,\pi),
$$
and
\begin{equation}\label{K_0}
K_0 := \frac{K}{2}\, \int_0^\pi  \frac{\sin \lambda}\lambda\, d\lambda.
\end{equation}
It is known that $K=(1+3^{-2}+5^{-2}+\ldots)/(1-3^{-2}+5^{-2}-\ldots)\approx  1.347$
(see \cite{Davis1}) and therefore $K_0 < 1.25$.

For $F\in\mathcal{S}_n(\mathbb{T})$, we introduce two functions
\begin{equation}\label{ell}
\ell_F := \log\det F - n \log_+ \|F\| \in L_1(\mathbb{T}) ,
\end{equation}
where $\log_+x=\max(0,\log x)$ (note that $\ell_F\leq 0$ a.e.), and
\begin{equation}\label{qf}
Q_F := \frac{\left(\max\left\{1, \|F\|\right\}\right)^n}{\det F}=e^{-\ell_F}=e^{|\ell_F|},
\end{equation}
and provide two estimates which can be used when information on the size of $\ell_F$ or $Q_F$ is available.
\begin{theorem}\label{Th.1.3}
Suppose $F \in \mathcal{S}_n(\mathbb{T}) \cap L_{p_0}(\mathbb{T})^{n\times n}$, $G \in \mathcal{S}_n(\mathbb{T})$, and $\ell_F \in L_{p_1}(\mathbb{T})$,  $p_0, p_1 \in (1, \infty)$, $\|G - F\|_{L_1}\leq 1$. Then for  any $\alpha \in (0, 1)$,
the following estimate holds
\begin{gather}\label{1.3.1}
\|G^+ - F^+\|_{H_2}^2 \le 4\|G - F\|_{L_1}
 + 2 q_0^{\frac{1}{q_0}} \left(\|F\|_{L_{p_0}} + 1\right)\\
\times\Big[c(p_0)\big\|\log_+ \|G\| - \log_+ \|F\|\big\|_{L_1}^\frac{1}{q_0} + 2(2p_0)^{\frac{1}{p_0}}
\Big(3n\left\|G - F\right\|_{L_1}^{1 - \alpha} \notag\\
+ 2(n+1) \alpha^{1 - p_1}p_1  \left\|\ell_F\right\|_{L_{p_1}}^{p_1} \big|\!\log\left\|G - F\right\|_{L_1}\!\big|^{1 - p_1} 
 + \Big\|\log\frac{\det G}{\det F}\Big\|_{L_1}\Big)^{\frac{1}{q_0}}\Big]  ,\notag
\end{gather}
where
\begin{equation}\label{cq}
c(p_0) := 2^{\frac{p_0 + 1}{p_0}} K_0^{\frac{1}{q_0}}  , \ \ \ q_0 := \frac{p_0}{p_0 - 1} .
\end{equation}
Furthermore, if in addition  $F \in L_{\infty}(\mathbb{T})^{n\times n}$, then the above inequality can be modified as
\begin{gather}\label{D8.3}
 \|G^+ - F^+\|_{H_2}^2 \le 4\|G - F\|_{L_1}
 + 4  \max\left\{\|F\|_{L_{\infty}}, 1\right\}\\
\times\Big(K_0\Big\|\log_+ \|G\| - \log_+ \|F\|\Big\|_{L_1}
 + 3n\left\|G - F\right\|_{L_1}^{1 - \alpha}  \nonumber\\
+2(n+1) \alpha^{1 - p_1}p_1  \left\|\ell_F\right\|_{L_{p_1}}^{p_1} \left|\log\left\|G - F\right\|_{L_1}\right|^{1 - p_1}
+ \Big\|\log\frac{\det G}{\det F}\Big\|_{L_1}\Big)  .\nonumber
\end{gather}
\end{theorem}

Example \ref{ex1}  in Section \ref{5} shows that one cannot completely dispose of the term $ |\log\left\|F-G\right\|_{L_1}|^{1 - p_1}$ in the
estimates in Theorem \ref{Th.1.3} or even improve the power $1 - p_1$ beyond $-p_1$, see \eqref{low}. It would be interesting to find the optimal power in this term.

\begin{theorem}\label{Th.1.4}
Suppose $F \in \mathcal{S}_n(\mathbb{T}) \cap L_{p_0}(\mathbb{T})^{n\times n}$, $G \in \mathcal{S}_n(\mathbb{T})$, and $Q_F \in L_{p_1}(\mathbb{T})$,  $p_0, p_1 \in (1, \infty)$.
If $\|F - G\|_{L_1} \le e^{-4}$,   then
\begin{gather}\label{1.4.1.}
 \|G^+ - F^+\|_{H_2}^2 \le 4\|G - F\|_{L_1}  \\
+ 2 q_0^{\frac{1}{q_0}} \left(\|F\|_{L_{p_0}} + 1\right)
\Big[c(p_0)\Big\|\log_+ \|G\| - \log_+ \|F\|\Big\|_{L_1}^\frac{1}{q_0}+ 2(2p_0)^{\frac{1}{p_0}}\notag \\
\times  \Big(\Big(3n +
\frac{4(n+1) \left\|Q_F\right\|_{L_{p_1}}^{2p_1}}{p_1 + 1}  \big|\log\left\|G - F\right\|_{L_1}\big|\Big) \left\|G - F\right\|_{L_1}^{\frac{p_1}{p_1 + 1}}
+ \Big\|\log\frac{\det G}{\det F}\Big\|_{L_1}\Big)^{\frac{1}{q_0}}\Big] ,\notag
\end{gather}
where $c(p_0)$ and $q_0$ are defined by \eqref{cq}.

Furthermore, if in addition  $F \in L_{\infty}(\mathbb{T})^{n\times n}$, then the above inequality can be modified as
\begin{gather*}
 \|G^+ - F^+\|_{H_2}^2 \le 4\|G - F\|_{L_1}
 + 4  \max\left\{\|F\|_{L_{\infty}}, 1\right\}
\Big[K_0\Big\|\log_+ \|G\| - \log_+ \|F\|\Big\|_{L_1}\\
+  \Big(3n + \frac{4(n+1) \left\|Q_F\right\|_{L_{p_1}}^{2p_1}}{p_1 + 1}
\left|\log\left\|G - F\right\|_{L_1}\right|\Big) \left\|G - F\right\|_{L_1}^{\frac{p_1}{p_1 + 1}}
+ \Big\|\log\frac{\det G}{\det F}\Big\|_{L_1}\Big]  .
\end{gather*}
\end{theorem}

Example \ref{ex2} in Section \ref{5} shows that the optimal power in the term
$$
\left\|G - F\right\|_{L_1}^{\frac{p_1}{p_1 + 1}}
$$
in Theorem \ref{Th.1.4}  cannot be higher than $\frac{2p_1}{2p_1 + 1} = 1 - \frac{1}{2p_1 + 1}$ (see \eqref{low'}). It would be interesting to find out how much one can improve the exponent $\frac{p_1}{p_1 + 1} = 1 - \frac{1}{p_1 + 1}$ in this term.

The estimates provided by Theorems 1.4 and 1.5 are of course more complicated than their counterparts in the scalar case given in Theorem \ref{Thrm.2}. Note, however, that the terms $\left\|\log\frac{\det G}{\det F}\right\|_{L_1}$ and $\Big\|\log_+ \|G\| - \log_+ \|F\|\Big\|_{L_1}$
correspond to $\|\log G - \log F\|_{L_1}$ and $\|\log_+ G - \log_+ F\|_{L_1} \le \|\log G - \log F\|_{L_1}$ in the scalar case. The term
$\Big\|\log_+ \|G\| - \log_+ \|F\|\Big\|_{L_1}$ can also be estimated by $\|G-F\|_{L_1}$ (see \eqref{elem} below). Since
$\left\|G - F\right\|_{L_1}^{1 - \alpha} = o\left( |\log\left\|G - F\right\|_{L_1}|^{1 - p_1}\right)$ as $\left\|G - F\right\|_{L_1} \to 0$ for any
$\alpha \in (0, 1)$ and $p_1 \in (1, \infty)$,
the main difference with the scalar case is the term containing $\log\left\|G - F\right\|_{L_1}$.
Examples \ref{ex1} and \ref{ex2} show the necessity of this term's presence and even proximity of its form to the optimum.

The above estimates are significantly simplified when matrix function $F$, as well as its inverse $F^{-1}$, are bounded.

\begin{theorem}\label{allbdd}
Let  $F \in \mathcal{S}_n(\mathbb{T})\cap L_\infty(\mathbb{T})^{n\times n}$ and
$\ell_F  \in L_\infty(\mathbb{T})$. Then
\begin{equation} \label{1.5.1}
\|G^+ - F^+\|_{H_2}^2 \le  \|F\|_{L_\infty} \left(n\,e^{\|\ell_F\|_{L_\infty}}  \left\|G - F\right\|_{L_1} +
\left\|\log\frac{\det G}{\det F}\right\|_{L_1}\right)  .
\end{equation}
\end{theorem}

Theorems \ref{Th.1.3}-\ref{allbdd} follow from more general results proved in Section 3 in the context of  Orlicz spaces (see Theorems 3.1-3.3).

At the end of the paper, we return to the scalar case and show that the exponent $\frac{p-1}{p}$ in estimate \eqref{estTh2} is optimal even if we allow the constant $C(p)$ to depend on $f$:

\begin{theorem}\label{1.6}
Let  $p>1$. For each $\gamma>\frac{p-1}{p}$, there exist  functions $f, f_k\in \mathcal{S}_1(\mathbb{T})\cap L_p(\mathbb{T})$, $k=1,2,\ldots$, such that
$$
\|f-f_k\|_{L_1}\to 0\;\; \text{ and }\;\; \|\log f-\log f_k\|_{L_1}\to 0,
$$
while
\begin{equation} \label{1.6.2}
\frac{\|f^+-f_k^+\|_{H_2}^2}{\|\log f-\log f_k\|_{L_1}^\gamma}\to\infty.
\end{equation}
\end{theorem}

\section{Notation and auxiliary results}

Let $\mathbb{D}=\{z\in\mathbb{C}:|z|<1\}$, and let $H_p=H_p(\mathbb{D})$, $p>0$, be the Hardy space of analytic functions:
$$
H_p=\{f\in\mathcal{A}(\mathbb{D}):\sup_{r<1}\int\nolimits_0^{2\pi}|f(re^{i\theta})|^p\,d\theta<\infty\}.
$$
A function $f\in H_p$ is called {\em outer} (we denote $f\in H_p^O$) if
$$
f(z)=c\cdot  \exp\left(\frac 1{2\pi}
\int\nolimits_0^{2\pi}\frac{e^{i\theta}+z}{e^{i\theta}-z}\log
\big|f(e^{i\theta})\big|\,d\theta\right),\;\;\;\;\;|c|= 1,\;\;z\in\mathbb{D} .
$$
A matrix function $F\in H_2(\mathbb{D})^{n\times n}$ is called outer if $\det F\in H_{2/n}^O$ (see, e.g. \cite{EL10}). Since factorization \eqref{MSF} implies that $\det F(t)=|\det F^+(t)|^2$, and the spectral factor $F^+$ is outer with positive definite $F^+(0)$, we have (see, e.g. \cite[Th.17.17]{Rud})
\begin{equation}\label{detFO}
\log\det F^+(0)=\int\nolimits_\mathbb{T}\log|\det F^+(t)|\,dm=\frac1{2}\int\nolimits_\mathbb{T}\log \det F(t)\,dm,
\end{equation}
where $m$ is the normalized Lebesgue measure on $\mathbb{T}=\partial\mathbb{D}$.

The norms of matrices and matrix functions are defined as in \eqref{nmm} and \eqref{nmf}.
For any matrix $A\in\mathbb{C}^{n\times n}$, $|A|$ denotes the nonnegative square root of $AA^*$, i.e. $|A|$ is positive semidefinite, $|A|\geq 0$, and $|A|^2=AA^*$. The relation $A\geq B$ means that $A-B\geq 0$.
Generalizing  $a\vee b:=\max(a,b)$ to self-adjoint matrices $A,B\in\mathbb{C}^{n\times n}$, we define
$$
A\vee B:=\frac12(A+B+|A-B|),
$$
which is the minimal upper bound for $A$ and $B$, i.e. i) $A,B\leq A\vee B$ and ii) $A,B\leq C\leq A\vee B$ implies that $C=A\vee B$ (see, e.g. \cite[\S 4]{Barclay}). Note that if
\begin{equation}\label{spd}
A=U\Lambda U^* \;\text{ with }\; \Lambda=\mathop{\rm diag}(\lambda_1,\lambda_2,\ldots,\lambda_n)
\end{equation}
 is a spectral decomposition of $A$ (obviously, then the norm defined by \eqref{nmm} is equal to $\|A\|=\max_{1\leq j\leq n}|\lambda_j|$) and $\eta\in\mathbb{R}$, then $A\vee\eta:=A\vee\eta I_n=U(\Lambda\vee\eta)U^*$, where $\Lambda\vee\eta=\mathop{\rm diag}(\lambda_1\vee\eta,\lambda_2\vee\eta,\ldots,\lambda_n\vee\eta)$.

 For a strictly positive definite matrix $A\in\mathbb{C}^{n\times n}$, the matrix $\log A\in\mathbb{C}^{n\times n}$ is defined by using the usual functional calculus, i.e. if $A$ has the form \eqref{spd}, then $\log A=U\mathop{\rm diag}(\log \lambda_1,\ldots,\log\lambda_n)U^*$ and $\operatorname{Tr}(\log A)=\log(\det A)$. Since $x-1\geq \log x$ for $x\in(0,\infty)$, we have for any $\eta>0$ (see \cite[p. 776]{Barclay})
 $$
 \frac{(x\vee \eta)-x}{\eta}=\left(1-\frac{x}{\eta}\right)\vee 0\leq\left(-\log\frac {x}{\eta}\right)\vee 0=\log(x\vee\eta)-\log x.
 $$
 Therefore
 \begin{equation}\label{A-log-nm}
\frac{\|A\vee\eta-A\|}{\eta}\leq \|\log (A\vee\eta)-\log A\|\leq \log\det(A\vee\eta)-\log\det A.
\end{equation}

We will use the following result on matrix spectral factorization proved in \cite{Barclay}:
\begin{proposition}\label{Prop}
{\rm (see \cite[Lemma 5.2]{Barclay})} Let $F,G\in \mathcal{S}_n(\mathbb{T})$ and let $0<\eta\leq \mathcal{R}<\infty$ be given constants. Suppose $F\geq\eta$ $($a.e.$)$ and let $\mathbb{P}$ be the projection-valued function defined by the functional calculus as $\mathbb{P}=\chi_{[0,\mathcal{R}]}(F)$, where $\chi_\Omega$ is the characteristic function of a set $\Omega$. Then the following estimate holds
\begin{equation}\label{Pr2.1}
\|(G^+-F^+)\mathbb{P}\|_{L_2}^2\leq \mathcal{R} n\left(\frac{\|G-F\|_{L_1}}{\eta}+2\left[1-\left(\frac{\det G^+(0)}{\det F^+(0)}\right)^{1/n}\right]\right).
\end{equation}
\end{proposition}

{\sc Remark.} Note that the right-hand side of \eqref{Pr2.1} differs slightly from the corresponding expression in Lemma 5.2 of \cite{Barclay}.
This is because the matrix norm $\|M\|$ used in \cite{Barclay} is different from \eqref{nmm}.

\smallskip

We need some notation from the theory of Orlicz  spaces (see \cite{orlicz2}, \cite{orlicz3}).
 Let $\Phi$ and
$\Psi$ be mutually complementary $N$-functions, i.e.
$$
\Phi(x)=\int_0^{|x|}u(\tau)\,d\tau \;\;\;\text{ and }\;\;\; \Psi(x)=\int_0^{|x|}v(\tau)\,d\tau,
$$
where $u:[0,\infty)\longrightarrow [0,\infty)$ is a right-continuous, nondecreasing function with
$u(0)=0$ and $u(\infty):=\lim_{\tau\to\infty}u(\tau)=\infty$, and $v$ is defined by the equality $v(x)=\sup_{u(\tau)\leq x}\tau$.
Let $(\Omega, \mathbb{S}, \mu)$ be a measure space, and let $L_\Phi(\Omega)$,
$L_\Psi(\Omega)$ be the corresponding Orlicz spaces, i.e.  $L_\Phi(\Omega)$ is the set of measurable functions on $\Omega$ for which either of the following  norms
$$
\|f\|_{\Phi} = \sup\left\{\left|\int_\Omega f g d\mu\right| : \
\int_\Omega \Psi(g) d\mu \le 1\right\}
$$
and
\begin{equation}\label{defnorm2}
\|f\|_{(\Phi)} = \inf\left\{\kappa > 0 : \
\int_\Omega \Phi\left(\frac{f}{\kappa}\right) d\mu \le 1\right\}
\end{equation}
is finite.  Note that these two norms are equivalent, namely  (see, e.g., \cite[(9.24)]{orlicz2} or \cite[\S 3.3, (14)]{orlicz3})
\begin{equation*}\label{Luxemburgequiv}
\|f\|_{(\Phi)} \le \|f\|_{\Phi} \le 2 \|f\|_{(\Phi)}\, , \ \ \ \forall f \in L_\Phi(\Omega) .
\end{equation*}
We will use the  H\"older inequality (see, e.g., \cite[(9.27)]{orlicz2} or \cite[\S 3.3, (16)]{orlicz3})
\begin{equation}\label{Hoel}
\left|\int_\Omega f g d\mu\right| \le \|f\|_{\Psi} \|g\|_{(\Phi)}
\end{equation}
as well as the following relations
\begin{align}
 &\; \Phi^{-1}(x_1+x_2)\leq \Phi^{-1}(x_1)+\Phi^{-1}(x_2) \;\text { for }\;x_1,x_2\geq 0;\label{1.20}\\
&\; 0<x_1\leq x_2\Longrightarrow \frac{\Phi(x_1)}{x_1}\leq \frac{\Phi(x_2)}{x_2};\label{1.18}\\
&\; x<\Phi^{-1}(x)\Psi^{-1}(x)\leq 2x\; \text{ for }\;x>0;\label{2.10}\\
&\; \|\chi_E\|_\Phi=\mu(E)\cdot \Psi^{-1}(1/\mu(E)),\label{9.11}\\
&\; \|\chi_E\|_{(\Phi)}=\frac1{\Phi^{-1}(1/\mu(E))}\; \text{ for }\;E\in\mathbb{S};\label{9.23}\\
&\|f\|_{(\Phi)}\leq\max\big\{1,\int_\Omega\Phi(f)\,d\mu\big\}; \label{J6}
\end{align}
which follow more or less directly from the definitions of the pair of functions $(\Phi,\Psi)$ and corresponding Orlicz norms (see \cite[formulas (1.20), (1.18), (2.10), (9.11), and (9.23)]{orlicz2}, and \cite[formula (6)]{Eugene}, respectively).

For a matrix function $F$ defined on $\mathbb{T}$, condition $F\in L_\Phi$ means that the function $t\mapsto \|F(t)\|$ belongs to $L_\Phi(\mathbb{T},\mathcal{B},\,dm)$. Slightly abusing the notation we assume that $\|F\|_\Phi$ is the $L_\Phi$-norm of this function.

Let $(\Phi, \Psi)$  be mutually complementary $N$-functions and define functions
\begin{equation}\label{22.5}
\Lambda_\Phi(s) := \inf\left\{\xi > 0 : \ \frac1\xi\, \Phi'\left(\frac1\xi\right) \le \frac1s\right\} , \ \ \ s > 0.
\end{equation}
and
\begin{equation}\label{MC2}
R_\Psi(\tau) := \tau \Psi^{-1}\left(\frac{4}{\tau}\right) , \ \ \ \tau > 0 .
\end{equation}

We use the functions $\Lambda_\Phi$ and $R_\Psi$ in the next sections. In an important special case \eqref{psiphi1} considered in Section \ref{4} , they are given by
an explicit formula \eqref{psiphi2}. Here, we only prove that these functions are of the same order of magnitude and that they
converge to $0$ when the argument tends to $0$ from the right.

Indeed, we have $\Lambda_\Phi(s)\leq\frac1\tau\Leftrightarrow\tau\Phi'(\tau)\leq\frac1s$.  It is easy to see that
\begin{equation}\label{MC23}
\tau \Phi'(\tau) \le \int_\tau^{2\tau} \Phi'(x)\, dx = \Phi(2\tau) - \Phi(\tau) < \Phi(2\tau) .
\end{equation}
Hence, taking $\tau=\frac12\Phi^{-1}\big(\frac1s\big)$ in \eqref{MC23}, we get
\begin{equation*}
\Lambda_\Phi(s) \le  \frac2{\Phi^{-1}\left(\frac1s\right)}\, , \ \ \ s > 0,
\end{equation*}
and therefore
\begin{equation}\label{Lb-0}
\Lambda_\Phi(s) \to 0 \ \mbox{ as } s \to 0+.
\end{equation}
On the other hand,
$$
\tau \Phi'(\tau) \ge \int_0^{\tau} \Phi'(x)\, dx = \Phi(\tau)
$$
and hence
\begin{equation*}
\Lambda_\Phi(s) \ge  \frac1{\Phi^{-1}\left(\frac1s\right)}\, , \ \ \ s > 0.
\end{equation*}
Using \eqref{1.20} and \eqref{2.10}, one gets
\begin{gather*}
R_\Psi(\tau) \le 4 \tau \Psi^{-1}\left(\frac{1}{\tau}\right) \le \frac8{\Phi^{-1}\left(\frac1\tau\right)} \le 8 \Lambda_\Phi(\tau) , \\
R_\Psi(\tau) \ge  \tau \Psi^{-1}\left(\frac{1}{\tau}\right) > \frac1{\Phi^{-1}\left(\frac1\tau\right)} \ge \frac12\, \Lambda_\Phi(\tau) .
\end{gather*}
Hence
\begin{equation}\label{RLambda}
\frac12\, \Lambda_\Phi(\tau) < R_\Psi(\tau) \le 8 \Lambda_\Phi(\tau) , \ \ \ \forall \tau > 0 ,
\end{equation}
and therefore,

\begin{equation}\label{R-0}
R_\Psi(\tau) \to 0 \ \mbox{ as } \tau \to 0+,
\end{equation}

The following lemma will be used in the sequel.

\begin{lemma}\label{Interp}
Let $(\Phi, \Psi)$ be a pair of mutually complementary $N$-functions. Then for any $u \in L_\infty(\mathbb{T})$
$$
\|u\|_{(\Phi)} \le \|u\|_{L_1} \Psi^{-1}\left(\frac{\|u\|_{L_\infty}}{\|u\|_{L_1}}\right).
$$
\end{lemma}
\begin{proof}
Let $\kappa > 0$. Since the function $\Phi$ is convex and
$\Phi(0) = 0$, one has $\Phi(\alpha x)\leq \alpha\Phi(x)$, for $0\leq\alpha\leq 1$. Therefore
$$
\Phi\left(\frac{|u(t)|}{\kappa}\right) \le \frac{|u(t)|}{\|u\|_{L_\infty}}\, \Phi\left(\frac{\|u\|_{L_\infty}}{\kappa}\right) .
$$
Hence
\begin{eqnarray*}
\int_{-\pi}^\pi \Phi\left(\frac{|u(e^{i\theta})|}{\kappa}\right)\, d\theta &\le&
\frac{\|u\|_{L_1}}{\|u\|_{L_\infty}}\, \Phi\left(\frac{\|u\|_{L_\infty}}{\kappa}\right)
\end{eqnarray*}
and if the right-hand side is less than 1, then $\|u\|_{(\Phi)}\leq \kappa$ because of the definition \eqref{defnorm2}. Therefore
\begin{eqnarray*}
&& \|u\|_{(\Phi)} \le \inf\left\{\kappa > 0 : \
\Phi\left(\frac{\|u\|_{L_\infty}}{\kappa}\right) \le \frac{\|u\|_{L_\infty}}{\|u\|_{L_1}}\right\} \\
&&= \inf\left\{\kappa > 0 : \
\frac{\|u\|_{L_\infty}}{\kappa} \le \Phi^{-1}\left(\frac{\|u\|_{L_\infty}}{\|u\|_{L_1}}\right)\right\} \\
&&= \frac{\|u\|_{L_\infty}}{\Phi^{-1}\left(\frac{\|u\|_{L_\infty}}{\|u\|_{L_1}}\right)} \le \|u\|_{L_1} \Psi^{-1}\left(\frac{\|u\|_{L_\infty}}{\|u\|_{L_1}}\right),
\end{eqnarray*}
the last inequality following from  \eqref{2.10}.
\end{proof}

The following theorems have been proved in \cite{ESS1} for the scalar spectral factorization:
\begin{theorem}\label{ESS.Th3}{\rm(\cite[Theorem 3]{ESS1})}
Let $f,g\in \mathcal{S}_1(\mathbb{T})$. For every pair $\Phi$ and
$\Psi$ of mutually complementary $N$-functions, the following estimate holds
\begin{equation*}\label{js2}
\|f^+ - g^+\|_{H_2}^2 \le 2\|f - g\|_{L_1} + 4 \|f\|_{\Psi}\, \Lambda_\Phi\left(\frac{K_0}2\, \|\log f - \log g\|_{L_1}\right)\, ,
\end{equation*}
where $K_0$ is defined by \eqref{K_0}.
\end{theorem}

\begin{theorem}\label{ESS.Cor2}{\rm(\cite[Corollary 2]{ESS1})}
Let $f,g\in \mathcal{S}_1(\mathbb{T})$ and $f\in L_\infty(\mathbb{T})$. Then
$$
\|f^+ - g^+\|_{H_2}^2 \le 2\|f - g\|_{L_1} + 2K_0 \|f\|_{L_\infty} \|\log f - \log g\|_{L_1}\, .
$$

\end{theorem}

\section{Estimates in terms of the Orlicz norms}

In this section, we provide upper estimates for $\|F^+-G^+\|_{H_2}$ involving Orlicz norms.
Theorems  \ref{Th.1.3}-\ref{allbdd} will be obtained as corollaries of Theorems \ref{Th.3.1}-\ref{ellbdd} proved in this section. We recall that functions $\Lambda_{\Phi}$ and $R_{\Psi}$ from the statements of these theorems are defined by \eqref{22.5} and $\eqref{MC2}$.
\begin{theorem}\label{Th.3.1}
Let $F,G\in \mathcal{S}_n(\mathbb{T})$, and let $(\Phi_0, \Psi_0)$ and $(\Phi_1, \Psi_1)$ be two pairs of mutually complementary $N$-functions such that
\begin{equation}\label{3.1.1.}
F \in L_{\Psi_0} \;\text{ and  } \;\ell_F  \in L_{\Psi_1}.
\end{equation}
 Then for any nondecreasing function $\nu : [0, \infty) \to  [0, 1]$ satisfying
\begin{equation}\label{275}
\nu(\tau) \to 0 \text{  and } \tau/\nu(\tau) \to 0 \text{ as } \tau \to 0+
\end{equation}
the following estimate holds
\begin{gather}  \label{Tm3.1.shp}
 \|G^+ - F^+\|_{H_2}^2 \le 4\|G - F\|_{L_1} \\
+ 2 \left(\|F\|_{\Psi_0} +\Phi_0^{-1}({1})\right)
\Big[4 \Lambda_{\Phi_0}\left(\frac{K_0}2\, \Big\|\log_+ \|G\| - \log_+ \|F\|\Big\|_{L_1}\right) \notag\\
+ R_{\Psi_0}\Bigg(\frac{6n\left\|G - F\right\|_{L_1}}{\nu\left(\left\|G - F\right\|_{L_1}\right)} +
\frac{4(n+1)|\log\nu\left(\left\|G - F\right\|_{L_1}\right)|}{\Psi_1\Big(\frac{|\log\nu\left(\left\|G - F\right\|_{L_1}\right)|}{\left\|\ell_F\right\|_{(\Psi_1)}}\Big)}
+ 2\left\|\log\frac{\det G}{\det F}\right\|_{L_1}\Bigg)\Big]  .\notag
\end{gather}
\end{theorem}

{\sc Remark.}
Since every integrable function belongs to a certain Orlicz  space (see \cite[\S 8]{orlicz2}), for each $F\in\mathcal{S}_n(\mathbb{T})$, there exist mutually complimentary $N$-functions $(\Phi_0,\Psi_0)$ and $(\Phi_1,\Psi_1)$ such that \eqref{3.1.1.} holds. Then it follows from
\eqref{Lb-0}, \eqref{R-0}, and the property $s/\Psi_1(s)\to 0$ as $s\to \infty$, that Theorem \ref{Th.3.1} proves the existence of an estimate of
the form \eqref{3.6.1} for every $F\in \mathcal{S}_n(\mathbb{T})$.

\begin{proof}
Let
\begin{equation}\label{M_F}
M_F(t) := \max\left\{1, \|F(t)\|\right\} \text{ and } \ \ F_1(\theta) := \frac{1}{M_F(t)}\, F(t) .
\end{equation}
It is clear that (see \eqref{ell})
\begin{equation}\label{ellneg}
0 \le F_1 \le 1 \ \mbox{ and } \  \ell_F = \log\det F_1 \le 0 .
\end{equation}

Suppose $\eta \in (0, 1)$. Following \cite[Section 6]{Barclay}, consider $\tilde{F}_1 := F_1 \vee \eta$.
Let
$$
0 \le \lambda_1(t) \le \lambda_2(t) \le \cdots \le \lambda_n(t) \le 1
$$
be the eigenvalues of $F_1(t)$. Then
\begin{equation}\label{elll}
|\ell_F(t)| = |\log\det F_1(t)| = |\log \lambda_1(t) + \cdots + \log \lambda_n(t)|
\ge  |\log \lambda_1(t)| .
\end{equation}
Let
$$
\Omega_\eta := \left\{t \in \mathbb{T} : \ \tilde{F}_1(t) \not= F_1(t)\right\} .
$$
Then $\Omega_\eta = \left\{t \in \mathbb{T} : \ \lambda_1(t) < \eta\right\}$, and (see \eqref{elll} and \eqref{9.23})
\begin{gather}\label{meas}
 |\log\eta| \left\|\chi_{\Omega_\eta}\right\|_{(\Psi_1)} \le \left\|\log \lambda_1\right\|_{(\Psi_1)}  \ \Longrightarrow \ \
\frac{1}{\Psi_1^{-1}\left(\frac{1}{m(\Omega_\eta)}\right)} \le \frac{\left\|\ell_F\right\|_{(\Psi_1)}}{|\log\eta|} \\
\Longrightarrow \ \ \frac{|\log\eta|}{\left\|\ell_F\right\|_{(\Psi_1)}} \le \Psi_1^{-1}\left(\frac{1}{m(\Omega_\eta)}\right)
\ \ \Longrightarrow \ \ m(\Omega_\eta) \le \left(\Psi_1\left(\frac{|\log\eta|}{\left\|\ell_F\right\|_{(\Psi_1)}}\right)\right)^{-1}\, .\nonumber
\end{gather}
Consequently, using the H\"older inequality \eqref{Hoel} for Orlicz spaces  and equality \eqref{9.11}, one gets
\begin{gather}\label{logs}
\int_\mathbb{T}(\log\det \tilde{F}_1 - \log\det F_1)\,dm=\int_\mathbb{T}(\log\det \tilde{F}_1 - \log\det F_1)\chi_{\Omega_\eta}\,dm
 \\
\leq -\int_\mathbb{T} \log\det F_1 \chi_{\Omega_\eta}\,dm
=\int_\mathbb{T}\left|\ell_F\right| \chi_{\Omega_\eta} \,dm \leq \left\|\ell_F\right\|_{(\Psi_1)} \left\|\chi_{\Omega_\eta}\right\|_{\Phi_1} =
  \notag \\
  \left\|\ell_F\right\|_{(\Psi_1)}  m(\Omega_\eta) \Psi_1^{-1}\left(\frac1{ m(\Omega_\eta)}\right)
 \le \left\|\ell_F\right\|_{(\Psi_1)} \frac{1}{\Psi_1\left(\frac{|\log\eta|}{\left\|\ell_F\right\|_{(\Psi_1)}}\right)} \frac{|\log\eta|}{\left\|\ell_F\right\|_{(\Psi_1)}}
= \frac{|\log\eta|}{\Psi_1\left(\frac{|\log\eta|}{\left\|\ell_F\right\|_{(\Psi_1)}}\right)}\, .\notag
\end{gather}
The last inequality above is obtained by taking $x_1=|\log\eta|/\|\ell_F\|_{(\Psi_1)}$, $x_2= \Psi_1^{-1}\left(1/m(\Omega_\eta)\right)$ in \eqref{1.18}, and applying inequality \eqref{meas}.

Since
\begin{equation} \label{logxn}
-\log x \ge n\left(1 - x^{1/n}\right)\;\; \text{ for all }x \in (0, 1],
\end{equation}
it follows from Proposition \ref{Prop} (note that in our situation $\mathbb{P}$ is the identity operator and one can take $\mathcal{R}=1$ in \eqref{Pr2.1}  since $\|F_1(t)\|\leq 1$ a.e.), \eqref{A-log-nm}, and \eqref{detFO} that
\begin{gather*}
\left\|\tilde{F}_1^+ - F_1^+\right\|_{H_2}^2 \le \frac{n\left\|\tilde{F}_1 - F_1\right\|_{L_1}}{\eta} + 2n \left[1 - \left(\frac{\det F_1^+(0)}{\det \tilde{F}_1^+(0)}\right)^{1/n}\right] \\
\leq n\int_\mathbb{T}(\log\det\tilde{F}_1-\log\det F_1)\,dm + 2\left|\log\left(\frac{\det F_1^+(0)}{\det \tilde{F}_1^+(0)}\right)\right|
\\=(n+1)\int_\mathbb{T}(\log\det\tilde{F}_1-\log\det F_1)\,dm \\
\end{gather*}
and
\begin{gather}\label{dm1}
\left\|\tilde{F}_1^+ - G_1^+\right\|_{H_2}^2 \le \frac{n\left\|\tilde{F}_1 - G_1\right\|_{L_1}}{\eta} +
2n \left[1 - \left(\frac{\det G_1^+(0)}{\det \tilde{F}_1^+(0)}\right)^{1/n}\right]  \\
\leq  \frac{n\left\|F_1 - G_1\right\|_{L_1}}{\eta} + \frac{n\left\|\tilde{F}_1 - F_1\right\|_{L_1}}{\eta} + 2\left|\log\left(\frac{\det G_1^+(0)}{\det \tilde{F}_1^+(0)}\right)\right| \leq \frac{n\left\|F_1 - G_1\right\|_{L_1}}{\eta} + \nonumber \\
  n\int_\mathbb{T}(\log\det\tilde{F}_1-\log\det F_1)\,dm +\left|\int_\mathbb{T}(\log\det\tilde{F}_1-\log\det G_1)\,dm\right|\leq  \frac{n\left\|F_1 - G_1\right\|_{L_1}}{\eta}  \nonumber \\
 + (n+1)\int_\mathbb{T}(\log\det\tilde{F}_1-\log\det F_1)\,dm +\left|\int_\mathbb{T}(\log\det{F}_1-\log\det G_1)\,dm\right|  \nonumber .
\end{gather}
Hence  \eqref{logs} implies
\begin{gather}\label{G1}
\left\|F_1^+ - G_1^+\right\|_{H_2}^2 \le 2 \left\|\tilde{F}_1^+ - F_1^+\right\|_{H_2}^2 + 2 \left\|\tilde{F}_1^+ - G_1^+\right\|_{H_2}^2 \leq   \frac{2n\left\|F_1 - G_1\right\|_{L_1}}{\eta}\\
 + 4(n+1)\int_\mathbb{T}(\log\det\tilde{F}_1-\log\det F_1)\,dm +
2 \left|\int_\mathbb{T}(\log\det{F}_1-\log\det G_1)\,dm\right| \nonumber \\
 \le \frac{2n\left\|F_1 - G_1\right\|_{L_1}}{\eta} + \frac{4(n+1)|\log\eta|}{\Psi_1\left(\frac{|\log\eta|}{\left\|\ell_F\right\|_{(\Psi_1)}}\right)} +
2\left|\int_\mathbb{T}(\log\det{F}_1-\log\det G_1)\,dm\right|.  \nonumber
\end{gather}
Elementary inequalities for any $a,b>0$
\begin{eqnarray}\label{elem}
\left|\max\{1, a\} - \max\{1, b\}\right| \le |a - b| , \ \ |\log_+a - \log_+b| \le |a - b| ,
\end{eqnarray}
and the estimates $M_F, M_G\geq 1$ (see \eqref{M_F}) imply
\begin{gather}
 \left\|F_1 - G_1\right\|_{L_1} = \left\|\frac{1}{M_F}F - \frac{1}{M_G}G\right\|_{L_1}
 =  \left\|\frac{M_G - M_F}{M_G}\frac{1}{M_F}\, F + \frac{1}{M_G}\, (F - G)\right\|_{L_1} \notag \\\label{303030}
 \le \left\|M_G - M_F\right\|_{L_1}  + \left\|F - G\right\|_{L_1}  \le 2 \left\|F - G\right\|_{L_1}
\end{gather}
and (see \eqref{ell} and \eqref{ellneg})
\begin{gather}\label{G1L1log2}
 \|\log\det F_1 - \log\det G_1\|_{L_1}\le \|\log\det F - \log\det G\|_{L_1}  \\
+ n\Big\|\log_+ \|F\| - \log_+ \|G\|\Big\|_{L_1}
 \le \|\log\det F - \log\det G\|_{L_1} + n\Big\|\|F\| - \|G\|\Big\|_{L_1} \nonumber \\
 \le \|\log\det F - \log\det G\|_{L_1} + n\|F - G\|_{L_1} . \nonumber
\end{gather}
Taking $\eta = \nu\left(\left\|F - G\right\|_{L_1}\right)$ in \eqref{G1} and applying \eqref{303030} and \eqref{G1L1log2}, one gets
\begin{gather}\label{first1}
\left\|F_1^+ - G_1^+\right\|_{H_2}^2 \le \frac{4n\left\|F - G\right\|_{L_1}}{\nu\left(\left\|F - G\right\|_{L_1}\right)} +
\frac{4(n+1)|\log\nu\left(\left\|F - G\right\|_{L_1}\right)|}{\Psi_1\left(\frac{|\log\nu\left(\left\|F - G\right\|_{L_1}\right)|}{\left\|\ell_F\right\|_{(\Psi_1)}}\right)} \\
+ 2\|\log\det F - \log\det G\|_{L_1} + 2n\|F - G\|_{L_1} .\nonumber
\end{gather}
Note that
$$
\|F_1^+(t)\| = \|F_1(t)\|^{1/2} \le 1
$$
and similarly for $G_1^+$. Hence
$$
u(t) := \|F_1^+(t) - G_1^+(t)\|^2 \le
\left(\|F_1^+(t)\| + \|G_1^+(t)\|\right)^2 \le 4 ,\;\;\|u\|_{L_1}=\left\|F_1^+ - G_1^+\right\|_{H_2}^2,
$$
and taking into account Lemma \ref{Interp} and  definition \eqref{MC2}, one gets
\begin{equation}\label{R1}
\Big\|\|F_1^+(\cdot) - G_1^+(\cdot)\|^2\Big\|_{(\Phi_0)} \le R_{\Psi_0}\left(\left\|F_1^+ - G_1^+\right\|_{H_2}^2\right).
\end{equation}

Applying the H\"older inequality \eqref{Hoel}, estimate  \eqref{R1}, Theorem \ref{ESS.Th3}, and properties \eqref{elem}, one obtains
\begin{gather}
\left\|F^+ - G^+\right\|_{H_2}^2 =  \left\|M_F^+ F_1^+ - M_G^+ G_1^+\right\|_{H_2}^2
\le 2\left\|M_F^+ (F_1^+ -  G_1^+)\right\|_{H_2}^2 \\
+ 2 \left\|(M_F^+  - M_G^+) G_1^+\right\|_{H_2}^2
\le 2\Big\|M_F \|F_1^+(\cdot) - G_1^+(\cdot)\|^2\Big\|_{L_1} + 2 \left\|M_F^+  - M_G^+\right\|_{H_2}^2 \nonumber
 \\ \le 2 \|M_F\|_{\Psi_0} R_{\Psi_0}\left(\left\|F_1^+ - G_1^+\right\|_{H_2}^2\right)
+ 4\|M_F - M_G\|_{L_1}\nonumber \\
+8 \|M_F\|_{\Psi_0}\, \Lambda_{\Phi_0}\left(\frac{K_0}2\, \|\log M_F - \log M_G\|_{L_1}\right) \nonumber \\
\le 2 \left(\|F\|_{\Psi_0} + \|1\|_{\Psi_0}\right) \Big[R_{\Psi_0}\left(\left\|F_1^+ - G_1^+\right\|_{H_2}^2\right) \label{comb} \nonumber \\
+ 4 \Lambda_{\Phi_0}\left(\frac{K_0}2\, \Big\|\log_+ \|F\| - \log_+ \|G\|\Big\|_{L_1}\right)\Big] \nonumber
+ 4\|F - G\|_{L_1} . \nonumber
\end{gather}
It remains now  to apply \eqref{first1} and \eqref{9.11} in order to get \eqref{Tm3.1.shp}.
\end{proof}

Suppose $\ell_F  \in L_{\Psi_1}$ and $\left\|\ell_F\right\|_{(\Psi_1)} > 1$.
If $\Psi_1$ does not satisfy the $\Delta_2$ condition, $\Psi(2x)\leq C\psi(x)$, $x>0$,  then
\begin{equation}\label{noDelta2}
\varrho\left(\ell_F; \Psi_1\right) := \int_\mathbb{T} \Psi_1(|\ell_F|)\, dm < \infty
\end{equation}
might not hold (see \cite[\S 4 and Theorem 8.2]{orlicz2}).
If the latter condition is satisfied, one has the following version of Theorem \ref{Th.3.1}.
\begin{theorem}\label{main'}
Let $F,G\in \mathcal{S}_n(\mathbb{T})$, and let $(\Phi_0, \Psi_0)$ and $(\Phi_1, \Psi_1)$ be two pairs of mutually complementary $N$-functions such that $F \in L_{\Psi_0}$ and
\eqref{noDelta2} holds. Suppose
\begin{equation}\label{PI}
\Pi_{\Psi_1}(\ell_F) := \left\|\ell_F\right\|_{(\Psi_1)} \max\{1, \varrho\left(\ell_F; \Psi_1\right)\} .
\end{equation}
Then for any nondecreasing function
$\nu : [0, \infty) \to  [0, 1]$ which satisfies \eqref{275}
the following estimate holds
\begin{gather}\label{3.2.1.}
 \|G^+ - F^+\|_{H_2}^2 \le  4\|G - F\|_{L_1}
 + 2 \left(\|F\|_{\Psi_0} +\Phi_0^{-1}({1})\right)\\
\times\Big[4 \Lambda_{\Phi_0}\Big(\frac{K_0}2\, \Big\|\log_+ \|G\| - \log_+ \|F\|\Big\|_{L_1}\Big)
 + R_{\Psi_0}\Big(\frac{6n\left\|G - F\right\|_{L_1}}{\nu\left(\left\|G - F\right\|_{L_1}\right)}  \notag\\
+ \frac{4(n+1) \Pi_{\Psi_1}(\ell_F) |\log\nu\left(\left\|G - F\right\|_{L_1}\right)|}{\Psi_1\left(|\log\nu\left(\left\|G - F\right\|_{L_1}\right)|\right)}
+ 2\left\|\log\frac{\det G}{\det F}\right\|_{L_1}\Big)\Big] .\notag
\end{gather}
\end{theorem}

{\em Proof.}\
One needs to make the following changes in the proof of Theorem \ref{Th.3.1}. The estimate \eqref{meas} is replaced by
\begin{equation}\label{meas'}
 |\log\eta|\, \chi_{\Omega_\eta} \le \left|\log \lambda_1\right|  \Rightarrow \ \
\Psi_1(|\log\eta|) m(\Omega_\eta) \le \int_\mathbb{T} \Psi_1(|\ell_F|)\, dm
\Rightarrow  m(\Omega_\eta) \le \frac{\varrho\left(\ell_F; \Psi_1\right)}{\Psi_1(|\log\eta|)}\, .
\end{equation}
Using this along with the facts that $\Psi_1^{-1}$ is a concave function with $\Psi_1^{-1}(0) = 0$, which implies that $\Psi_1^{-1}(\alpha x)\geq \alpha\Psi_1^{-1}(x)$ for $0\leq\alpha\leq1$, and that $\Psi_1^{-1}(x)/x$ is a decreasing function for $x>0$, one gets the following analogue of \eqref{logs}:
\begin{eqnarray}\label{logs'}
&& \int_\mathbb{T}(\log\det \tilde{F}_1 - \log\det F_1)\,dm \le
\left\|\ell_F\right\|_{(\Psi_1)}  m(\Omega_\eta) \Psi_1^{-1}\left(\frac{1}{ m(\Omega_\eta)}\right)   \\
&& \le \left\|\ell_F\right\|_{(\Psi_1)} \frac{\varrho\left(\ell_F; \Psi_1\right)}{\Psi_1(|\log\eta|)}
\Psi_1^{-1}\left(\frac{\Psi_1(|\log\eta|)}{\varrho\left(\ell_F; \Psi_1\right)}\right)   \nonumber \\
&& \le\ \left\|\ell_F\right\|_{(\Psi_1)} \frac{\varrho\left(\ell_F; \Psi_1\right)}{\Psi_1(|\log\eta|)}
\frac{1}{\min\{1, \varrho\left(\ell_F; \Psi_1\right)\}}
\Psi_1^{-1}\left(\Psi_1(|\log\eta|)\right)   \nonumber \\
&&= \left\|\ell_F\right\|_{(\Psi_1)} \max\{1, \varrho\left(\ell_F; \Psi_1\right)\}\,  \frac{|\log\eta|}{\Psi_1(|\log\eta|)}\, =\Pi_{\Psi_1}(\ell_F)\,  \frac{|\log\eta|}{\Psi_1(|\log\eta|)}.\nonumber
\end{eqnarray}
The rest of the proof is the same as for Theorem \ref{Th.3.1}. In particular, the inequality \eqref{first1} takes the form
\begin{eqnarray}\label{first21}
\left\|F_1^+ - G_1^+\right\|_{H_2}^2 \le \frac{4n\left\|F - G\right\|_{L_1}}{\nu\left(\left\|F - G\right\|_{L_1}\right)} +
\frac{4(n+1)\Pi_{\Psi_1}(\ell_F)|\log\nu\left(\left\|F - G\right\|_{L_1}\right)|}{\Psi_1\left({|\log\nu\left(\left\|F - G\right\|_{L_1}\right)|}\right)} \nonumber \\
+ 2\|\log\det F - \log\det G\|_{L_1} + 2n\|F - G\|_{L_1}.   \hskip+2cm\Box
\end{eqnarray}
\smallskip

Next we consider the case where $\ell_F\in L_\infty$. This holds, in particular, when $F, F^{-1}\in L_\infty(\mathbb{T})^{n\times n}$.

\begin{theorem}\label{ellbdd}
Let $F,G\in \mathcal{S}_n(\mathbb{T})$, and let $(\Phi_0, \Psi_0)$ be a pair of mutually complementary $N$-functions. Suppose $F \in L_{\Psi_0}$ and
$\ell_F  \in L_\infty$. Then
\begin{eqnarray*}
&& \|G^+ - F^+\|_{H_2}^2 \le 4\|G - F\|_{L_1} \\
&& + 2 \left(\|F\|_{\Psi_0} +\Phi_0^{-1}({1})\right)
\Big[4 \Lambda_{\Phi_0}\left(\frac{K_0}2\, \Big\|\log_+ \|G\| - \log_+ \|F\|\Big\|_{L_1}\right) \\
&&+  R_{\Psi_0}\Big(\left(2 e^{\|\ell_F\|_{L_\infty}} + 1\right)n \left\|G - F\right\|_{L_1} +
\left\|\log\frac{\det G}{\det F}\right\|_{L_1}\Big)\Big]  .
\end{eqnarray*}
\end{theorem}
\begin{proof}
One needs to make the following changes in the proof of Theorem \ref{Th.3.1}. It follows from \eqref{elll} that
\begin{gather}\label{Slowest}
 |\log \lambda_1(t)|  \le \|\ell_F\|_{L_\infty} \ \Longrightarrow \  \log \lambda_1(t)  \ge -\|\ell_F\|_{L_\infty}
\\ \Longrightarrow
\lambda_1(t)  \ge e^{-\|\ell_F\|_{L_\infty}}
\ \Longrightarrow \ F_1 \ge e^{-\|\ell_F\|_{L_\infty}} . \nonumber
\end{gather}
Taking $\eta =  e^{-\|\ell_F\|_{L_\infty}}$ in the proof of Theorem \ref{Th.3.1}, then $\tilde{F}_1 = F_1$ and, by virtue of \eqref{dm1}, \eqref{detFO}, \eqref{303030}, and \eqref{G1L1log2}, one gets
\begin{eqnarray*}
&& \left\|F_1^+ - G_1^+\right\|_{H_2}^2 = \left\|\tilde{F}_1^+ - G_1^+\right\|_{H_2}^2 \le \frac{n}{\eta}\left\|\tilde{F}_1 - G_1\right\|_{L_1} +
2\left|\log\left(\frac{\det G_1^+(0)}{\det \tilde{F}_1^+(0)}\right)\right|
\\
&& = n\left\|F_1 - G_1\right\|_{L_1} e^{\|\ell_F\|_{L_\infty}}  + \left|\int_\mathbb{T}(\log\det{F}_1-\log\det G_1)\,dm\right| \\
&& \le 2n \left\|F - G\right\|_{L_1} e^{\|\ell_F\|_{L_\infty}}  +  \|\log\det F - \log\det G\|_{L_1} + n\|F - G\|_{L_1} \\
&&= \left(2 e^{\|\ell_F\|_{L_\infty}} + 1\right)n \left\|F - G\right\|_{L_1} +
\|\log\det F - \log\det G\|_{L_1} .
\end{eqnarray*}
The rest of the proof is the same as in the proof of Theorem \ref{Th.3.1}.
\end{proof}

\section{Proofs of Theorems 1.3-1.5}\label{4}

In this section, we derive the theorems formulated in the introduction as corollaries of the corresponding estimates obtained in the previous section.

\smallskip

{\em Proof of Theorem \ref{Th.1.3}}. For the first statement, one has to take
\begin{equation}\label{psiphi1}
\Psi_j(t) =\frac{t^{p_j}}{p_j};\;  \Phi_j(t) = \frac{t^{q_j}}{q_j},\;   q_j = \frac{p_j}{p_j - 1},
\;j = 0, 1, \;\text{ and } \nu(\tau) = \min(\tau^\alpha,1)
\end{equation}
in Theorem \ref{Th.3.1}. Then it is easy to see that
$\Psi_j^{-1}(\tau)=(p_j\tau)^{1/p_j}$, $\Phi_j^{-1}(\tau)=(q_j\tau)^{1/q_j}$. In particular,
\begin{equation}\label{535}
\Phi^{-1}_0(1)=q_0^\frac1{q_0},
\end{equation}
and (see \eqref{22.5} and \eqref{MC2})
\begin{equation}\label{psiphi2}
\Lambda_{\Phi_0}(s) = s^{1/q_0} \text{ and } R_{\Psi_0}(s) = (4{p_0})^{\frac1{p_0}} s^{\frac{{p_0} -1}{{p_0}}}=(4{p_0})^{\frac1{p_0}} s^{\frac{1}{{q_0}}},
\end{equation}
while
\begin{equation}\label{psiphi3}
\|\cdot\|_{\Psi_0} = q_0^{\frac{1}{q_0}} \|\cdot\|_{L_{p_0}} , \ \ \ \|\cdot\|_{(\Psi_1)} =  \|\cdot\|_{L_{p_1}},
\end{equation}
see \cite[(9.7)]{orlicz2}. Therefore \eqref{1.3.1} follows from \eqref{Tm3.1.shp} using the following equalities:
$\log\nu(\|G-F\|_{L_1})=\alpha\log\|G-F\|_{L_1}$, \eqref{psiphi2}, \eqref{psiphi3}, and \eqref{535}.

To prove the second statement, one has to use the estimate
\begin{equation}\label{D8.1}
\Big\|M_F \|F_1^+(\cdot) - G_1^+(\cdot)\|^2\Big\|_{L_1}\leq \|M_F\|_{L_\infty} \left\|F_1^+ - G_1^+\right\|_{H_2}^2 = \max\left\{\|F\|_{L_{\infty}}, 1\right\} \left\|F_1^+ - G_1^+\right\|_{H_2}^2,
\end{equation}
while the estimate
\begin{equation}\label{D8.2}
\|M_F^+ - M_G^+\|_{H_2}^2 \le 2\|M_F - M_G\|_{L_1} + 2K_0 \|M_F\|_{L_\infty} \|\log M_F - \log M_G\|_{L_1}
\end{equation}
is given by Theorem \ref{ESS.Cor2}. We have $\log M_F=\log_+\|F\|$ (see \eqref{M_F}), and
$$|M_F-M_G| =|\max(1,\|F\|)-\max(1,\|G\|)|\leq\big|\|F\|-\|G\|\big|\leq \|F-G\|$$
(see \eqref{elem}). Therefore, \eqref{D8.2} can be rewritten as
\begin{equation}\label{D8.25}
\|M_F^+ - M_G^+\|_{H_2}^2 \le 2\|F - G\|_{L_1} + 2K_0 \max(\|F\|_{L_\infty},1) \big\|\log_+ \|F\| - \log_+\|G\|\big\|_{L_1}.
\end{equation}

Inequality \eqref{first1} takes the form
\begin{gather}\label{first2}
\left\|F_1^+ - G_1^+\right\|_{H_2}^2 \le {4n\left\|F - G\right\|_{L_1}}^{1-\alpha} +
4(n+1)\alpha^{1-p_1}p_1\left\|\ell_F\right\|_{L_{p_1}}^{p_1}\\
\times  \big|\log\left(\left\|F - G\right\|_{L_1}\right)\big|^{1-p_1} + 2\|\log\det F - \log\det G\|_{L_1} + 2n\|F - G\|_{L_1} \notag
\end{gather}
after the corresponding substitution of
$$\Psi_1\left(\frac{|\log\nu\left(\left\|F - G\right\|_{L_1}\right)|}{\left\|\ell_F\right\|_{(\Psi_1)}}\right)=
\frac1{p_1}\left(\frac{\alpha|\log\left(\left\|F - G\right\|_{L_1}\right)|}{\left\|\ell_F\right\|_{L_{p_1}}}\right)^{p_1}.$$

It follows from \eqref{comb}, \eqref{D8.1}, and \eqref{D8.25} that
\begin{gather}\label{D8.4}
\left\|F^+ - G^+\right\|_{H_2}^2 \le 2\Big\|M_F \|F_1^+(\cdot) - G_1^+(\cdot)\|^2\Big\|_{L_1} + 2 \left\|M_F^+  - M_G^+\right\|_{H_2}^2 \leq \\
2\max\left\{\|F\|_{L_{\infty}}, 1\right\}\big( \left\|F_1^+ - G_1^+\right\|_{H_2}^2+ 2K_0 \big\|\log_+ \|F\| - \log_+\|G\|\big\|_{L_1}\big)+
4\|F - G\|_{L_1} \notag
\end{gather}
and \eqref{D8.3} now  follows from \eqref{D8.4} and \eqref {first2}.\hfill$\Box$

\smallskip

{\em Proof of Theorem \ref{Th.1.4}}. For the first statement, one has to take $\Psi_0$ and $\Phi_0$ the same as in \eqref{psiphi1},
 $\Psi_1(\tau):=\mathcal{A}(p_1\tau)$, where $\mathcal{A}(\tau):=e^\tau-\tau-1$, and
 \begin{equation}\label{nu4}
 \nu(\tau) := \min(1,\tau^{\frac{1}{p_1 + 1}})
 \end{equation}
 in Theorem \ref{main'}. Since
$$
\frac{p_1}{p_1 + 1}\, \left|\log \left\|F - G\right\|_{L_1}\right| \ge \frac{4p_1}{p_1 + 1} \ge 2 ,
$$
it follows from the inequality
$$
\mathcal{A}(\tau) \ge \frac12\, e^\tau , \ \ \ \tau \in [2,\infty) ,
$$
that
 \begin{equation}\label{22.5.1}
\Psi_1\left(|\log\nu\left(\left\|F - G\right\|_{L_1}\right)|\right) = \mathcal{A}\left(\frac{p_1}{p_1 + 1}\, \left|\log \left\|F - G\right\|_{L_1}\right|\right)
\ge \frac12\, \left\|F - G\right\|_{L_1}^{-\frac{p_1}{p_1 + 1}} .
 \end{equation}
Further (see \eqref{noDelta2} and \eqref{qf}),
$$
\max\{1, \varrho\left(\ell_F; \Psi_1\right)\} \le \int_{\mathbb{T}} e^{p_1 |\ell_F|}\, dm = \int_{\mathbb{T}} (Q_F)^{p_1}\, dm =
 \left\|Q_F\right\|_{L_{p_1}}^{p_1}.
$$
Hence (see \eqref{PI}, \eqref{J6})
 \begin{equation}\label{22.5.2}
\Pi_{\Psi_1}(\ell_F) = \left\|\ell_F\right\|_{(\Psi_1)} \max\{1, \varrho\left(\ell_F; \Psi_1\right)\} \le \max\{1, \varrho\left(\ell_F; \Psi_1\right)\}^2 \le
\left\|Q_F\right\|_{L_{p_1}}^{2p_1}.
 \end{equation}
 Therefore \eqref{1.4.1.} follows from \eqref{3.2.1.} using  \eqref{535}, \eqref{nu4}, \eqref{psiphi2}, \eqref{psiphi3}, \eqref{22.5.1}, and \eqref{22.5.2}.

 The second part of the statement can be proved in the same way as in the proof of Theorem \ref{Th.1.3}. In particular \eqref{first21} takes the form
 \begin{gather*}
 \left\|F_1^+ - G_1^+\right\|_{H_2}^2 \le  \left(4n +
\frac{8 (n+1)\left\|Q_F\right\|_{L_{p_1}}^{2p_1}}{p_1 + 1}  \big|\log\left\|G - F\right\|_{L_1}\big|\right) \left\|G - F\right\|_{L_1}^{\frac{p_1}{p_1 + 1}} \\
+ 2\|\log\det F - \log\det G\|_{L_1} + 2n\|F - G\|_{L_1}
 \end{gather*}
 after substitution of \eqref{22.5.1} and \eqref{22.5.2}, and the remaining steps are exactly the same.   \hfill$\Box$

 \smallskip

{\em Proof of Theorem \ref{allbdd}}. It follows from the last inequality of \eqref{Slowest} in the proof of Theorem \ref{ellbdd} that
$$
F(\vartheta) =   \max\left\{1, \|F(\vartheta)\|\right\} F_1(\vartheta) \ge e^{-\|\ell_F\|_{L_\infty}} .
$$
Taking again $\eta =  e^{-\|\ell_F\|_{L_\infty}}$ as in that proof, one gets $\tilde{F} := F \vee \eta = F$.
If $\mathcal{R} = \|F\|_{L_\infty}$, then the projection $P$ in Proposition \ref{Prop} equals the identity operator. Therefore,
$$
\left\|G^+ - F^+\right\|_{H_2}^2 \le \|F\|_{L_\infty} \left(n\,e^{\|\ell_F\|_{L_\infty}} \left\|G - F\right\|_{L_1} +
2n \left[1 - \left(\frac{\det G^+(0)}{\det F^+(0)}\right)^{1/n}\right]\right),
$$
and \eqref{1.5.1} follows from \eqref{logxn} and \eqref{detFO}.

 \section{Examples}\label{5}

 In this section we construct specific examples which show to what extent estimates obtained in Theorems \ref{Th.1.3} and \ref{Th.1.4} can possibly be improved.

 First we show that the exponent $1-p_1$ of the term $ \left(\log\left\|F-G\right\|_{L_1}\right)^{1 - p_1}$ in Theorem \ref{Th.1.3} cannot be improved beyond $-p_1$.

\begin{example}\label{ex1}
{\rm Let $1 < p_1 < \infty$, \ $\varepsilon, \delta \in (0, 1)$,  and let $\lambda_j$, $j = 0, 1, 2, 3$ be  outer functions such that $\lambda_j(0) > 0$ and
\begin{eqnarray*}
&& \left|\lambda_0\left(e^{i\vartheta}\right)\right| = \left\{\begin{array}{cl}
\exp\left(-\vartheta^{-1/p_1}\right) ,    &  \vartheta \in [\varepsilon, 2\varepsilon]  \\ \\
1 ,     &   \vartheta \in [-\pi, \pi)\setminus [\varepsilon, 2\varepsilon] ,
\end{array}\right. \\ \\
&& \left|\lambda_1\left(e^{i\vartheta}\right)\right| = \left\{\begin{array}{cl}
\sqrt{2}\, \exp\left(-\vartheta^{-1/p_1}\right) ,    &  \vartheta \in [\varepsilon, 2\varepsilon]  \\ \\
 \frac{1}{\sqrt{1 - \delta^2}}\, ,     &   \vartheta \in [-\pi, \pi)\setminus [\varepsilon, 2\varepsilon] ,
\end{array}\right. \\ \\
&& \left|\lambda_2\left(e^{i\vartheta}\right)\right| = \left\{\begin{array}{cl}
\frac1{\sqrt{2}}\, ,    &  \vartheta \in [\varepsilon, 2\varepsilon]  \\ \\
 \sqrt{1 - \delta^2}\, ,     &   \vartheta \in [-\pi, \pi)\setminus [\varepsilon, 2\varepsilon] ,
\end{array}\right. \\ \\
&& \left|\lambda_3\left(e^{i\vartheta}\right)\right| = \left\{\begin{array}{cl}
\frac1{\sqrt{2}}\, ,    &  \vartheta \in [\varepsilon, 2\varepsilon]  \\ \\
\delta ,     &   \vartheta \in [-\pi, \pi)\setminus [\varepsilon, 2\varepsilon] .
\end{array}\right.
\end{eqnarray*}
Let
\begin{equation} \label{+21}
F^+ := \begin{pmatrix}
   \lambda_0   &  0  \\
    0  &  1
\end{pmatrix} , \ \ \
G^+ := \begin{pmatrix}
   \lambda_1   &  0  \\
    \lambda_3  &  \lambda_2
\end{pmatrix}.
\end{equation}
These matrices are the spectral factors of
$$
F := F^+ \left(F^+\right)^\ast = \begin{pmatrix}
   |\lambda_0|^2   &  0  \\
    0  &  1
\end{pmatrix} , $$ {and}$$
G := G^+ \left(G^+\right)^\ast =  \begin{pmatrix}
 |\lambda_1|^2   &    \lambda_1\, \overline{\lambda_3}  \\
  \lambda_3\, \overline{ \lambda_1}   &  |\lambda_3|^2 +  |\lambda_2|^2
\end{pmatrix} =\begin{pmatrix}
 |\lambda_1|^2   &    \lambda_1\, \overline{\lambda_3}  \\
  \lambda_3\, \overline{ \lambda_1}   &  1
\end{pmatrix} .
$$
We have
\begin{equation} \label{+21+}
G - F = \begin{pmatrix}
 |\lambda_1|^2  -  |\lambda_0|^2 &    \lambda_1\, \overline{\lambda_3}  \\
  \lambda_3\, \overline{ \lambda_1}   &  0
\end{pmatrix} .
\end{equation}
Furthermore,
\begin{eqnarray*}
2\pi\left\||\lambda_1|^2  -  |\lambda_0|^2\right\|_{L_1} = \int_\varepsilon^{2\varepsilon} \exp\left(-2\vartheta^{-1/p_1}\right)\, d\vartheta
+ \int_{\vartheta \in [-\pi, \pi)\setminus [\varepsilon, 2\varepsilon]} \frac{\delta^2}{1 - \delta^2}\, d\vartheta \\
\le \varepsilon \exp\left(-2^{1 - 1/p_1}\varepsilon^{-1/p_1}\right) + \frac{2\pi \delta^2}{1 - \delta^2}\, ,
\end{eqnarray*}
and
\begin{eqnarray*}
2\pi\left\|\lambda_1\, \overline{\lambda_3}\right\|_{L_1} = \int_\varepsilon^{2\varepsilon} \exp\left(-\vartheta^{-1/p_1}\right)\, d\vartheta
+ \int_{\vartheta \in [-\pi, \pi)\setminus [\varepsilon, 2\varepsilon]} \frac{\delta}{\sqrt{1 - \delta^2}}\, d\vartheta \\
\le \varepsilon \exp\left(-2^{-1/p_1}\varepsilon^{-1/p_1}\right) + \frac{2\pi\delta}{\sqrt{1 - \delta^2}}\, .
\end{eqnarray*}
Therefore, taking $\delta = \frac{1}{4\pi} \varepsilon \exp\left(-2^{-1/p_1}\varepsilon^{-1/p_1}\right)$ in the above inequalities, one gets for sufficiently small $\varepsilon$
$$
2\pi\left\|G - F\right\|_{L_1} \le 4 \varepsilon \exp\left(-2^{-1/p_1}\varepsilon^{-1/p_1}\right) \le \exp\left(-(2\varepsilon)^{-1/p_1}\right) .
$$
Hence
\begin{equation} \label{21}
\varepsilon \ge \frac{\left|\log\|G - F\|_{L_1}\right|^{-p_1}}{2}\, .
\end{equation}
On the other hand (see \eqref{+21}),
\begin{equation} \label{22}
\|G^+ - F^+\|_{H_2}^2 \ge \|\lambda_3\|_{H_2}^2 \ge \frac{\varepsilon}{4\pi} .
\end{equation}
Thus, it follows from \eqref{21} and \eqref{22} that
\begin{equation}\label{low}
\|G^+ - F^+\|_{H_2}^2 \ge \frac{\left|\log\|G - F\|_{L_1}\right|^{-p_1}}{8\pi},
\end{equation}
which proves our claim.

Note that $\det G =  \det F=|\lambda_1|^2  |\lambda_2|^2 =  |\lambda_0|^2 $, $ 0 \le F \le 1$, and}
\begin{eqnarray*}
\|\ell_F\|_{L_{p_1}}= \left\|\log\det F\right\|_{L_{p_1}} = \left\|\log|\lambda_0|^2\right\|_{L_{p_1}} = \left(\frac1{2\pi}\int_\varepsilon^{2\varepsilon} 2^{p_1} \vartheta^{-1}\, d\vartheta\right)^{1/p_1}
= 2 \left(\frac{\log 2}{2\pi}\right)^{1/p_1}.
\end{eqnarray*}

\end{example}

\smallskip

Next we show that the exponent $\frac{p_1}{p_1 + 1}$ of the term $\left\|F-G\right\|_{L_1}^{\frac{p_1}{p_1 + 1}}$ in Theorem \ref{Th.1.4} cannot be changed to the value bigger than
$\frac{2p_1}{2p_1 + 1}$.

\begin{example}\label{ex2}
{\rm We change only the definitions of functions $\lambda_0$ and $\lambda_1$ in \eqref{+21}:
\begin{eqnarray*}
&& \left|\lambda_0\left(e^{i\vartheta}\right)\right| = \left\{\begin{array}{cl}
\vartheta^{\frac{1}{2p_1}} ,    &  \vartheta \in [\varepsilon, 2\varepsilon]  \\ \\
1 ,     &   \vartheta \in [-\pi, \pi)\setminus [\varepsilon, 2\varepsilon] ,
\end{array}\right. \\ \\
&& \left|\lambda_1\left(e^{i\vartheta}\right)\right| = \left\{\begin{array}{cl}
\sqrt{2}\, \vartheta^{\frac{1}{2p_1}} ,    &  \vartheta \in [\varepsilon, 2\varepsilon]  \\ \\
 \frac{1}{\sqrt{1 - \delta^2}}\, ,     &   \vartheta \in [-\pi, \pi)\setminus [\varepsilon, 2\varepsilon] .
\end{array}\right.
\end{eqnarray*}
One has
\begin{eqnarray*}
2\pi\left\||\lambda_1|^2  -  |\lambda_0|^2\right\|_{L_1} = \int_\varepsilon^{2\varepsilon} \vartheta^{1/p_1}\, d\vartheta
+ \int_{\vartheta \in [-\pi, \pi)\setminus [\varepsilon, 2\varepsilon]} \frac{\delta^2}{1 - \delta^2}\, d\vartheta \\
\le \frac{p_1}{p_1 + 1} \left(2^{\frac{p_1 + 1}{p_1}} - 1\right) \varepsilon^{\frac{p_1 + 1}{p_1}} + \frac{2\pi \delta^2}{1 - \delta^2}\, , \\
2\pi\left\|\lambda_1\, \overline{\lambda_3}\right\|_{L_1} = \int_\varepsilon^{2\varepsilon} \vartheta^{\frac{1}{2p_1}}\, d\vartheta
+ \int_{\vartheta \in [-\pi, \pi)\setminus [\varepsilon, 2\varepsilon]} \frac{\delta}{\sqrt{1 - \delta^2}}\, d\vartheta \\
\le \frac{2p_1}{2p_1 + 1} \left(2^{\frac{2p_1 + 1}{2p_1}} - 1\right) \varepsilon^{\frac{2p_1 + 1}{2p_1}}  + \frac{2\pi\delta}{\sqrt{1 - \delta^2}}\, .
\end{eqnarray*}
Taking $\delta = \varepsilon^{\frac{2p_1 + 1}{2p_1}}$, one gets for sufficiently small $\varepsilon$
$$
\left\|G - F\right\|_{L_1} \le C_{p_1}\, \varepsilon^{\frac{2p_1 + 1}{2p_1}}
$$
with a constant $C_{p_1} < \infty$ depending only on $p_1$.
Hence
\begin{equation} \label{28}
\varepsilon \ge \left(\frac{\|G - F\|_{L_1}}{C_{p_1}}\right)^{\frac{2p_1}{2p_1 + 1}}\, .
\end{equation}
On the other hand,
\begin{equation} \label{29}
\|G^+ - F^+\|_{H_2}^2 \ge \|\lambda_3\|_{H_2}^2 \ge \frac{\varepsilon}{4\pi} .
\end{equation}
Thus, it follows from \eqref{28} and \eqref{29} that
\begin{equation}\label{low'}
\|G^+ - F^+\|_{H_2}^2 \ge \mbox{const } \|G - F\|_{L_1}^{\frac{2p_1}{2p_1 + 1}}
\end{equation}
and the claim of the  example holds.

Note that $\det G =  \det F$, $ 0 \le F \le 1$, as in Example \ref{ex1} and}
\begin{eqnarray*}
\left\| Q_F\right\|_{L_{p_1}}= \left\|(\det F)^{-1}\right\|_{L_{p_1}} = \left\||\lambda_0|^{-2}\right\|_{L_{p_1}} \leq
 \left(\frac{\log 2}{2\pi}\right)^{1/p_1} +1.
\end{eqnarray*}
\end{example}

  \section{The scalar case}

  In this last section, we return to the discussion of the scalar spectral factorization and demonstrate that the power $\frac{p-1}{p}$ in  estimate \eqref{estTh2} is optimal. This is exactly the claim of Theorem \ref{1.6}, the proof of which follows.

Take any $\gamma >\gamma_0> \frac{p - 1}{p}$ (without loss of generality, assume that $\gamma_0<1$) and choose small $\alpha, \beta > 0$ such that
\begin{equation}\label{abg}
\gamma_0 \ge \frac{p - 1 + \beta}{(1 - \alpha)(p + \beta)}\, .
\end{equation}
The function
$$
\zeta \mapsto \omega(\zeta) := -i\, \frac{\zeta - 1}{\zeta + 1}
$$
maps the unit disk conformally onto the upper half-plane, $\omega(1) = 0$, and we have the equality $\omega(e^{i\vartheta})=\tan\frac\vartheta2$ for the boundary values.

Consider the principle branch of a function
$$
z \mapsto z^{\alpha - 1} ,
$$
which is analytic in the upper half-plane and positive on the positive half-line. It is easy to see that the function
$$
\zeta \mapsto w(\zeta) := -i (\omega(\zeta))^{\alpha - 1} = -i\left(-i\, \frac{\zeta - 1}{\zeta + 1}\right)^{\alpha - 1}
$$
belongs to the Hardy space $H_1(\mathbb{D})$, and
\begin{eqnarray}\label{150}
&& \mbox{Re}\, w\left(e^{i\vartheta}\right) = \left\{\begin{array}{cl}
-\left|\tan\frac{\vartheta}{2}\right|^{\alpha - 1} \sin(\pi\alpha) ,    &  -\pi < \vartheta < 0  \\ \\
0 ,     &   0 \le \vartheta < \pi ,
\end{array}\right.
\end{eqnarray}
\begin{eqnarray*}\nonumber
&& \mbox{Im}\, w\left(e^{i\vartheta}\right) = \left\{\begin{array}{cl}
\left|\tan\frac{\vartheta}{2}\right|^{\alpha - 1} \cos(\pi\alpha) ,    &  -\pi < \vartheta < 0 \\ \\ \notag
 -\left(\tan\frac{\vartheta}{2}\right)^{\alpha - 1} ,     &   0 \le \vartheta < \pi .
\end{array}\right.
\end{eqnarray*}
In particular
\begin{equation} \label{Rew}
\|\mbox{Re}\, w\|_{L_1}<\infty.
\end{equation}
Then
$$
\widetilde{\mbox{Re}\, w}\left(e^{i\vartheta}\right) = \mbox{Im}\, w\left(e^{i\vartheta}\right) + v_0 ,
$$
where
$$
v_0 := -\int_{-\pi}^\pi \mbox{Im}\, w\left(e^{i\vartheta}\right)\, d\vartheta > 0 .
$$
Let
\begin{equation} \label{151}
h_\varepsilon := \varepsilon\, \mbox{Re}\, w,
\end{equation}
 where $\varepsilon \in (0, 1)$ is a small parameter.
An easy calculation shows that if
\begin{equation} \label{vtht}
\vartheta \in I_\varepsilon := \left[2\arctan\left(\left(3 \pi + \varepsilon v_0\right)^{\frac{1}{\alpha - 1}} \varepsilon^{\frac{1}{1 - \alpha}}\right),\
2\arctan\left(\left(\pi + \varepsilon v_0\right)^{\frac{1}{\alpha - 1}} \varepsilon^{\frac{1}{1 - \alpha}}\right)\right]=:[\vartheta_1,\vartheta_2] ,
\end{equation}
then $\widetilde{h_\varepsilon}\left(e^{i\vartheta}\right) = \varepsilon \widetilde{\mbox{Re}\, w}\left(e^{i\vartheta}\right) \in
\left[-3 \pi, -\pi\right]$, and hence
\begin{equation}\label{cos}
\cos\left(\frac12\, \widetilde{h_\varepsilon}\left(e^{i\vartheta}\right)\right) \le 0\; \text{ for arbitrary } \vartheta \in I_\varepsilon .
\end{equation}
Note that
\begin{equation}\label{eps}
|I_\varepsilon| \ge
C_\alpha\, \varepsilon^{\frac{1}{1 - \alpha}}
\end{equation}
for sufficiently small $\varepsilon > 0$ and a suitable constant $C_\alpha > 0$ since
\begin{equation}\label{arctan}
\lim_{x\to 0+}\frac{\arctan x}{x}=1.
\end{equation}

Take a sufficiently large $\tau > 0$ such that
\begin{equation}\label{add0}
\frac{\tau + 1}{\tau + 2 - \alpha} > \gamma_0
\end{equation}
and let
\begin{equation} \label{ff1}
f\left(e^{i\vartheta}\right) := \left\{\begin{array}{cl}
|\vartheta|^\tau ,    &  -\pi < \vartheta < 0  \\
\vartheta^{-\frac{1}{p + \beta}} ,     &   0 \le \vartheta < \pi.
\end{array}\right.
\end{equation}
Then
$$
\|f\|_{L_p} =(2\pi)^{-\frac{1}{p}} \left(\frac{p + \beta}{\beta}\, \pi^{\frac{\beta}{p + \beta}} + \frac{1}{\tau p + 1}\, \pi^{\tau p + 1}\right)^{\frac{1}{p}} .
$$
Suppose
\begin{equation} \label{add1}
g_\varepsilon := f \exp\left(h_\varepsilon\right),
\end{equation}
where $h_\varepsilon$ is defined by \eqref{151}.
The functions $f$ and $g_\varepsilon$ coincide for $0 \le \vartheta < \pi$ (see the definitions \eqref{151} and \eqref{150}), and (see \eqref{Rew})
\begin{equation} \label{logep}
\|\log f - \log g_\varepsilon\|_{L_1} = \|h_\varepsilon\|_{L_1} = \varepsilon \|\mbox{Re}\, w\|_{L_1} =C\, \varepsilon\, \text{ with }C<\infty.
\end{equation}
Let $\rho := \frac{1}{\tau + 2 - \alpha}$. Then, applying the inequality $1-e^x\leq|x|$, \eqref{150} and \eqref{arctan}, we get
\begin{gather}
2\pi\|f - g_\varepsilon\|_{L_1} = \int_{-\pi}^0 |\vartheta|^\tau \left(1 - \exp\left(\varepsilon \mbox{Re}\, w\right)\right)\, d\vartheta
\le \pi^\tau \int_{-\pi}^{-\varepsilon^\rho} \varepsilon \left|\mbox{Re}\, w\right|\, d\vartheta +
\int_{-\varepsilon^\rho}^0 |\vartheta|^\tau\, d\vartheta \notag \\\label{Est1}
\le \pi^{\tau + 1} \varepsilon \left(\tan\frac{\varepsilon^\rho}{2}\right)^{\alpha - 1} \sin(\pi\alpha) + \frac{1}{\tau + 1} \varepsilon^{\rho(\tau + 1)}
\leq C_\tau \varepsilon^{\frac{\tau + 1}{\tau + 2 - \alpha}}\, .
\end{gather}
Further, by virtue of \eqref{add1}, \eqref{vtht},  \eqref{cos}, \eqref{ff1}, \eqref{eps}, and \eqref{abg}, we have
\begin{gather}\label{Est2}
\int_{-\pi}^\pi f\left(e^{i\vartheta}\right)
\left(1 - \cos\Big(\frac{1}{2}\, \left(\widetilde{\log f}\left(e^{i\vartheta}\right) - \widetilde{\log g}_\varepsilon\left(e^{i\vartheta}\right)\right) \Big)\right)\, d\vartheta \\
=\int_{-\pi}^\pi f\left(e^{i\vartheta}\right)
\left(1 - \cos\left(\frac12\, \widetilde{h_\varepsilon}\left(e^{i\vartheta}\right)\right)\right)\, d\vartheta
\ge \int_{I_\varepsilon} f\left(e^{i\vartheta}\right)\, d\vartheta\geq \min_{\vartheta\in I_\varepsilon}f\left(e^{i\vartheta}\right)\cdot |I_\varepsilon| \notag\\
= \vartheta_2^{-\frac1{p+\beta}}|I_\varepsilon|\ge C_{p, \alpha}\, \varepsilon^{-\frac{1}{(1 - \alpha)(p + \beta)}} \varepsilon^{\frac{1}{1 - \alpha}} =
C_{p, \alpha}\, \varepsilon^{\frac{p - 1 + \beta}{(1 - \alpha)(p + \beta)}} \ge C_{p, \alpha}\, \varepsilon^{\gamma_0}\notag
\end{gather}
for sufficiently small $\varepsilon > 0$ and a suitable constant $C_{p, \alpha} > 0$.

The inequality
\begin{gather*}
\|f^+ - g_\varepsilon^+\|_{H_2}^2\\
\geq \frac1{\pi} \int_{-\pi}^\pi f\left(e^{i\vartheta}\right)
\left(1 - \cos\Big(\frac{1}{2}\, \left(\widetilde{\log f}\left(e^{i\vartheta}\right) - \widetilde{\log g_\varepsilon}\left(e^{i\vartheta}\right)\right) \Big)\right)\, d\vartheta - 4 \|f - g_\varepsilon\|_{L_1}
\end{gather*}
was obtained in the proof of Theorem 1 in \cite{ESS1}. Hence, it follows from \eqref{Est1}, \eqref{Est2},  \eqref{add0}, and \eqref{logep} that
$$
\|f^+ - g_\varepsilon^+\|_{H_2}^2  \ge   \frac1\pi C_{p, \alpha}\, \varepsilon^{\gamma_0} -  \frac2\pi C_\tau\, \varepsilon^{\frac{\tau + 1}{\tau + 2 - \alpha}} \\
= C_{p, \alpha,\tau}\,\varepsilon^{\gamma_0}
\ge\frac{ C_{p, \alpha,\tau}}{C^{\gamma_0}} \|\log f - \log g_\varepsilon\|_{L_1}^{\gamma_0}
$$
for sufficiently small $\varepsilon > 0$.

If we take now a sequence of positive numbers $\varepsilon_k$ that converges to $0$ and suppose $f_k=g_{\varepsilon_k}$, then we get \eqref{1.6.2} since $\gamma>\gamma_0$.

 \section{Acknowledgments}

The third author was supported in part by Faculty Research funding from the Division of Science and Mathematics, New York University Abu Dhabi.


\def\cprime{$'$}
\providecommand{\bysame}{\leavevmode\hbox to3em{\hrulefill}\thinspace}
\providecommand{\MR}{\relax\ifhmode\unskip\space\fi MR }
\providecommand{\MRhref}[2]{%
  \href{http://www.ams.org/mathscinet-getitem?mr=#1}{#2}
}
\providecommand{\href}[2]{#2}

\end{document}